\documentclass [11pt] {article}
\usepackage{amsmath}
\usepackage{amsfonts}
\usepackage{amssymb}
\usepackage{amsfonts}
\usepackage{amssymb}
\usepackage{times}
\usepackage{color}
\usepackage{graphicx}

\newcommand{\Hom}{\mathrm{Hom}}

\newtheorem{lemma}{Lemma}[section]
\newtheorem{corollary}{Corollary}[section]
\newtheorem{theorem}{Theorem}

\newtheorem{remark}{Remark}[section]

\newtheorem{proposition}{Proposition}[section]
\newtheorem{definition}{Definition}[section]
\def\deg{\mbox{deg}\,}

\def\<{\langle}
\def\>{\rangle}

\def\deg{\mbox{deg}\,}

\begin{document}

\title{An information theoretic approach to Sidorenko's conjecture}
\author{{Bal\'azs Szegedy}}

\maketitle

\abstract{We investigate the famous conjecture by Erd\H os-Simonovits and Sidorenko using information theory. 
Our method gives a unified treatment for all known cases of the conjecture and it implies various new results as well. Our topological type conditions allow us to extend Sidorenko's conjecture to large families of $k$-uniform hypergraphs. This is somewhat unexpected since the conjecture fails for $k$ uniform hypergraphs in general.
\bigskip

\section{introduction}

In 1993 \cite{Sid} Sidorenko rased the question if for every bipartite graph $H=(\{1,2,\dots,n\},E)$ and bounded symmetric non-negative function $h$ on $[0,1]^2$ the following correlation inequality holds
\begin{equation}\label{conj}
\int \prod_{(i,j)\in E}h(x_i,x_j)~d\mu^n\geq \Bigl(\int h~d\mu^2\Bigr)^{|E|}.
\end{equation}
%(Note that in \cite{Sid} a non symmetric version was considered.)
The integrals on the left hand side of (\ref{conj}) arise as Mayer integrals in statistical mechanics, Feynman integrals in quantum field theory, and multicenter integrals in quantum chemistry. Furthermore they arise as homomorphism densities in the so-called graph limit theory.

Another, more combinatorial formulation of the conjecture, that turns out to be equivalent with yet another form stated independently by Erd\H os and Simonovits, is the following. For two finite graphs $H$ and $G$, a function $f:V(H)\rightarrow V(G)$ is called a homomorphism if it maps edges to edges. let $t(H,G)$ denote the probability that a random map from $V(H)$ to $V(G)$ is a homomorphism. 
Then, for the graph $H$, (\ref{conj}) is known to be equivalent with the statement that 

\begin{equation}\label{conj2}
t(H,G)\geq t(e,G)^{|E(H)|} 
\end{equation}

holds for every graph $G$ where $e$ is a single edge.

The conjecture is proven for numerous special families of bipartite graphs (Blakley, Roy \cite{BR}, Sidorenko \cite{Sid}, Benjamini, Peres \cite{BP}, Hatami \cite{Hat}, Conlon, Sudakov, Fox \cite{CFS}, Lov\'asz \cite{L}, Li, Szegedy \cite{LSz}, Kim, Lee and Lee \cite{KLL}). These results were obtained by using a variety of methods from combinatorics, probability theory, graph limit theory and even linear algebra. In this article we provide a new information theoretic approach which yields the conjecture for a class of graphs that contains all previous classes and many new graphs as well. Our class of graphs is defined as line graphs of certain higher dimensional complexes. Even if Sidorenko's conjecture fails in general (as it is the case for $k$-uniform hypergraphs if $k>2$) it is a natural objective to characterize all graph $H$ that satisfy it. Our results hint at a topological phenomenon that underlies the complete classification of graphs and hypergraphs satisfying the conjecture. 

 %Note that every graph $H$ represents a type of inequality that has to be satisfied by infinitely many graphs (or functions) and thus it is not even clear how to check the conjecture for a fix small graph $H$.
  It is an interesting fact that the first class of graphs satisfying the conjecture was discovered by Blakley and Roy already in 1965 when they proved it for paths. Their statement was formulated in a linear algebraic language. A large class of graphs satisfying the conjecture was discovered by Conlon, Fox and Sudakov. They proved in \cite{CFS} that bipartite graphs in which one point is complete to the other side satisfy the conjecture. It was discovered by Li and the author in \cite{LSz} that the result by Conlon, Fox and Sudakov has a short analytic proof based on Jensen's inequality applied for the functions $\log x$ and $x \log x$. This motivates us to use information theory as a general approach to the conjecture.

To explain our main results we define a family $\mathfrak{S}$ of bipartite graphs that may be of independent interest. A graph $H$ is in $\mathfrak{S}$ if there is a scheme for producing a probability distribution on the copies of $H$ in an arbitrary graph $G$ using a sequence of conditionally independent couplings starting from random edges in $G$. More precisely if $H$ is the single edge $e$ then the only allowed scheme is the uniform distribution on $\Hom(e,G)$. Assume that we have such schemes for $H_1$ and $H_2$ i.e. probability distributions $\mu_1(G)$ on $\Hom(H_1,G)\subset V(G)^{V(H_1)}$ and $\mu_2(G)$ on $\Hom(H_2,G)\subset V(G)^{V(H_2)}$ for every $G$. Assume furthermore that for every graph $G$ the marginal distribution of $\mu_1(G)$ on some set $S_1\subset V(H_1)$ is the same as the marginal distribution of $\mu_2(G)$ on $S_2\subset V(H_2)$ using some bijection between $S_1$ and $S_2$.
Then we can take the conditional independent coupling of $\mu_1(G)$ and $\mu_2(G)$ over this joint marginal to obtain a new probability scheme $G\mapsto\mu_3(G)$. The new scheme is defined on $\Hom(H_3,G)$ where $H_3$ is obtained by taking the disjoint union of $H_1$ and $H_2$ and then identifying $S_1$ and $S_2$ using the bijection.
The class $\mathfrak{S}$ consists of those graphs $H$ that admit such a scheme.  Random walks and branching random walks on $G$ are special cases of this framework.
As a demonstration of our method we will prove the next theorem.

\medskip

{\bf Theorem}~~{\it If there is a probability scheme for $H$ built up in a way that all gluing operations use subsets that span forests then $H$ satisfies Sidorenko's conjecture.}

\medskip

A more precise formulation can be found in theorem \ref{forestgluing}. This theorem itself includes bipartite graphs in which one point is complete to the other side, tree-arrangeable graphs, bipartite graphs in which one side has size at most $4$ vertices, hypercubes up to dimension $5$ and many more graphs. To go further we need to develop a background theory for working with iterated conditionally independent couplings.  For this purpose we introduce reflection complexes. Refelection complexes are combintorial structures with a topological flavor. They encode the construction of probability distributions of their frames in an arbitrary graph $G$. We introduce the notion of thick graphs (definition \ref{defthick}) as line graphs (one dimensional frames) of reflection complexes satisfying a linear algebraic condition. Our main theorem for graphs is the following.

\medskip

\noindent{\bf Theorem}~~{\it Thick graphs satisfy Sidorenko's conjecture}

\medskip

See also theorem \ref{thicksid}. Thick graphs generalize the idea of theorem \ref{forestgluing} and they contain all known exaples for Sidorenko's conjecture. In particular thick graphs are closed with respect to a certain subdivision operation in which we replace the edges of a thick graph by another thick graph using spanned forests as vertices (see theorem \ref{subdiv}). A special case of this operation is the $\square$-product with a tree studied in \cite{KLL}. We show that if $H$ is a thick graph and $T$ is a tree then $H\square T$ is also thick. This result shows that the examples constructed in \cite{KLL} for Sidorenko's conjecture (including high dimensional grids and hypercubes \cite{Hat}) are thick graphs.    
%To help the reader we give an example for such a scheme. Let $\mu_e$ denote the uniform probability distribution on the directed edges in $G$. We can view $\mu_e$ as a coupling of two copies of $\mu_v$ where $\mu_v$ is the stationary distibution of the random walk on $G$. That is, $\mu_v$ is a probability distibution on the vertex set of $G$ in which the probability of a vertex is proportional to its degree. Now if we take two copies of $\mu_e$ then we have four such marginal distributions identical with $\mu_v$.  

Below we briefly explain how information theory enters our argument. Let ${\rm Hom}(H,G)$ denote the set of homomorphisms from $H$ to $G$. We have  that the quantity $d(H,G):=-\ln t(H,G)$ is equal to the relative entropy (also called Kullback-Leibler divergence and not to confuse with conditional entropy) of the uniform measure on ${\rm Hom}(H,G)$ with respect to the uniform measure $\nu$ on all functions $f:V(H)\rightarrow V(G)$. Thus (\ref{conj2}) is equivalent with the entropy inequality $d(H,G)\leq |E(H)|d(e,G)$. It is another fact that the uniform distribution has the smallest relative entropy with respect to $\nu$ among all probability distributions on ${\rm Hom}(H,G)$. It follows that if we manage to find another probability measure $\mu$ (witness measure) on ${\rm hom}(H,G)$ for every $G$ whose relative entropy is not greater than $|E(H)|d(e,G)$ then the Sidorenko conjecture is proved for $H$. This is the motivation to construct probability measures on ${\rm Hom}(H,G)$ that are easier to analyze than the uniform measure. We will build up such measures by iterating conditionally independent couplings. An advantage of this is that relative entropy satisfies an inclusion-exclusion type formula for conditionally independent couplings and thus it gives a method to understand the relative entropy of the measure that we build up this way. 

Or methods can be generalized to hypergraphs. We show that, despite of the fact that Sidorenko's conjecture fails for $k$-uniform hypergraphs (see \cite{Sid}) if $k>2$, there are large families of $k$-uniform hypergraphs satisfying the conjecture. In particular we prove a hypergraph analogue of the famous Bakley-Roy inequality (see \cite{BR}).

Finally we mention that all our results work in the non-symmetric (multipartite) setting. Our statements and proofs require only minor modifications to achieve this.

\section{Relative entropy and conditionally independent couplings}

In this chapter we review some basic facts about relative entropy and couplings. Let $\mu$ and $\nu$ be two probability measures on the same $\sigma$-algebra such that $\mu$ is absolutely continuous with respect to $\nu$. The {\bf relative entropy} function $D(\mu\parallel\nu)$ is equal to $\mathbb{E}_\mu(\log(d\mu/d\nu))$. If $X$ is a finite set with probability measures $\mu$ and $\nu$ then 
$$D(\mu\parallel\nu)=\sum_{x\in X}(\log\mu(x)-\log\nu(x))\mu(x)$$
where the summand is defined to be $0$ whenever $\mu(x)$ is zero. Note that the absolute continuity of $\mu$ means that $\nu(x)=0$ implies $\mu(x)=0$. Assume that $\mu$ is concentrated on some subset $Y\subseteq X$. Then by Jensen's inequality applied for the function $z\mapsto z\log z$ one obtains that
\begin{equation}\label{inex} 
D(\mu\parallel\nu)\geq-\log(\nu(Y))
\end{equation}
with equality if and only if $\mu(y)=\nu(y)/\nu(Y)$ holds for every $y\in Y$. In particular 
\begin{equation}\label{pos}
D(\mu\parallel\nu)\geq 0
\end{equation}
 holds for every $\mu$ and $\nu$.

Let $\{(X_i,\mu_i)\}_{i=1}^3$ be three finite probability spaces. Assume that $\{\psi_i:X_i\rightarrow X_3\}_{i=1,2}$ are measure preserving maps. Then we say that $X_3$ is a {\bf joint factor} of $(X_1,\mu_1)$ and $(X_2,\mu_2)$. Note that the measure on $X_3$ is uniquely determined by $\psi_1$ (or $\psi_2$) since $\mu_3(A)=\mu_1(\psi_1^{-1}(A))$ holds for $A\subseteq X_3$. Let $X_4$ denote the set of elements $(x_1,x_2)$ in $X_1\times X_2$ satisfying $\psi_1(x_1)=\psi_2(x_2)$ and that $\mu_3(\psi_1(x_1))\neq 0$. A measure $\mu$ on $X_4$ is called a {\bf coupling} of $(X_1,\mu_1)$ and $(X_2,\mu_2)$ over the joint factor $X_3$ if the projections $\pi_1:X_4\rightarrow X_1$ and $\pi_2:X_4\rightarrow X_2$ are measure preserving on $(X_4,\mu)$.

Let $\mu_4$ be the measure on $X_4$ defined by 

\begin{equation}\label{inex0}
\mu_4((x_1,x_2))=\frac{\mu_1(x_1)\mu_2(x_2)}{\mu_3(\psi_1(x_1))}.
\end{equation}

 It is clear that the projections $\{\pi_i:X_4\rightarrow X_i\}_{i=1,2}$ are measure preserving.
We say that $X_4$ together with the maps $\pi_1$ and $\pi_2$ is the {\bf conditionally independent coupling} of $X_1$ and $X_2$ over the joint factor $X_3$.

Keeping the above notation, assume that there are other measures $\nu_i$ on the sets $X_i$ for $1\leq i\leq 4$ such that $(X_4,\nu_4)$ is the conditionally independent coupling of $(X_1,\nu_1)$ and $(X_2,\nu_2)$ over the joint factor $(X_3,\nu_3)$ with the same maps $\psi_1$ and $\psi_2$. 
Using (\ref{inex0}) we get the next inclusion-exclusion type formula.
\begin{equation}\label{inex1}
D(\mu_4\parallel\nu_4)= D(\mu_1\parallel\nu_1)+D(\mu_2\parallel\nu_2)-D(\mu_3\parallel\nu_3).
\end{equation}

The next lemma says that, among couplings, the conditionally independent coupling minimizes the relative entropy with respect to a conditionally independent coupling.

\begin{lemma}\label{nagyobb} Keeping the above notation and assumptions let $\mu$ be a coupling of $(X_1,\mu_1)$ and $(X_2,\mu_2)$ over the joint factor $X_3$. Then
$$D(\mu\parallel\nu_4)\geq D(\mu_4\parallel\nu_4).$$
\end{lemma}

\begin{proof} First we argue that $D(\mu\parallel\nu_4)-D(\mu_4\parallel\nu_4)=H(\mu_4)-H(\mu)$ where $H$ denotes the usual entropy. To see this it is enough to show that $\sum_{x\in X_4}\log(\nu_4(x))(\mu(x)-\mu_4(x))=0$. Decomposing the sum into three sums according to $\log(\nu_4(x))=\log(\nu_1(\pi_1(x))+\log(\nu_2(\pi_2(x))-\log(\nu_3(\psi_1(\pi_1(x))))$ and using that $\mu$ is a coupling of $(X_1,\mu_1)$ and $(X_2,\mu_2)$ over $(X_3,\mu_3)$ (and so the marginals of $\mu$ and $\mu_4$ coincide) we obtain the desired equation. From $H(\mu_4)=H(\mu_1)+H(\mu_2)-H(\mu_3)$ we get that 
$$D(\mu\parallel\nu_4)-D(\mu_4\parallel\nu_4)=H(\mu_1)+H(\mu_2)-H(\mu)-H(\mu_3)$$ and the right hand side is positive by Shannon's inequality for entropy.
\end{proof}

\medskip

\section{Probability distributions of graph homomorphisms}

Recall that $\Hom(H,G)\subset V(G)^{V(H)}$ denotes the set of homomorphisms from $H$ to $G$ and $t(H,G)$ denotes the probability that a random map $f:V(H)\rightarrow V(G)$ is a homomorphism. We interpret $\Hom(H,G)$ as the set of copies of $H$ in $G$ and $t(H,G)$ as the density of $H$ in $G$.
Let $\tau(H,G)$ denote the uniform distribution on $\Hom(H,G)$ and let $\nu(H,G)$ denote the uniform distribution on $V(G)^{V(H)}$. Let us use the convention that $D(\mu):=D(\mu\parallel\nu(H,G))$ for an arbitrary probability distribution $\mu$ on $V(G)^{V(H)}$. It is clear that $D(\tau(H,G))=-\log(t(H,G))$ holds for every $H$ and $G$. This creates the connection between subgraph densities and relative entropy. 

We will use use the following notation. In a graph $G$ let $\kappa$ denote the probability distribution on the vertices in which the probability of a vertex is proportional to its degree. The role of $\kappa$ for us is that it is the distribution of an end point of a uniformly chosen random edge. Let us use the short hand notation $D_v:=D(\kappa)$ and $D_e=D(\tau(e,G))$. In this paper the edge set of a target graph $G$ is always assumed to be not empty. This guarantees that the distributions $\tau(e,G)$ and $\kappa$ exist. 
As we pointed out in the introduction, Sidorenko's conjecture for $H$ is equivalent with the statement that

\begin{equation}\label{entsid}
D(\tau(H,G))\leq |E(H)|~D_e
\end{equation}
holds for all graphs $G$.
Note that any probability distribution $\mu$ on $\Hom(H,G)$ satisfies that
\begin{equation}\label{jen}
D(\tau(H,G))\leq D(\mu).
\end{equation}
If $\mu$ satisfies $D(\mu)\leq |E(H)|D_e$ then we will say that $\mu$ is a {\bf witness} measure. It follows from (\ref{jen}) that if $\mu$ is a witness measure then $H$ satisfies the Sidorenko conjecture in $G$.
 
In this chapter we will study probability distributions on homomorphism sets that are iteratively obtained from the uniform distribution on edges using conditionally independent couplings. We will use factors of very specific form. Assume that $\mu$ is a probability distribution on $\Hom(H,G)$ and let $\beta:S\rightarrow V(H)$ be an injective map (labeling) for some set $S$.
Then the map $\phi\rightarrow\phi\circ\beta$ (where $\phi\in\Hom(H,G)$) on $\Hom(H,G)$ defines a factor of $(\Hom(H,G),\mu)$. We denote this factor by $(V(G)^S,\mu|_\beta)$ and call it a {\bf vertex factor} of $\mu$. If $S\subseteq V(H)$ is a subset of $V(H)$ then we denote by $\mu|_S$ the probability measure $\mu|_\beta$ where $\beta:S\rightarrow S$ is the identity map. If $S$ is empty then $\mu|_S$ is defined on a single point $V(G)^0$ and $D(\mu|_S)=0$.

Assume that we have two probability spaces $(\Hom(H_1,G),\mu_1)$ and $(\Hom(H_2,G),\mu_2)$ and two injective maps $\{\beta_i:[n]\rightarrow V(H_i)\}_{i=1,2}$ such that $\mu_3:=\mu_1|_{\beta_1}=\mu_2|_{\beta_2}$. Then we denote by $C(\mu_1,\mu_2,\beta_1,\beta_2)$ the conditionally independent coupling of $\mu_1$ and $\mu_2$ over $\mu_3$.

\medskip

A {\bf probability scheme} of a graph $H$ is a function $f$ on the set of finite graphs whose value $f(G)$ is a probability distribution on $\Hom(H,G)$. We say that $H$ is the {\bf frame} of the probability scheme $f$. Let $f_i$ be a probability scheme for $H_i$ where $i=1,2$. Assume that $\{\beta_i:[n]\rightarrow V(H_i)\}_{i=1,2}$ are two labelings such that $f_1(G)|_{\beta_1}=f_2(G)|_{\beta_2}$ holds for every $G$. Then we say that $\beta_1$ and $\beta_2$ define a joint vertex factor of $f_1$ and $f_2$. The conditionally independent coupling $g=C(f_1,f_2,\beta_1,\beta_2)$ of $f_1$ and $f_2$ is the function $g$ whose value on $G$ is $C(f_1(G),f_2(G),\beta_1,\beta_2)$. The frame of $g$ is the graph obtained by identifying the vertices with the same label in the disjoint union of $H_1$ and $H_2$. After identification we delete multiple edges.

\begin{definition} Let $\mathfrak{A}$ denote the smallest set of probability schemes which contains the scheme $G\rightarrow \tau(e,G)$ (uniform random edge) and is closed with respect to conditionally independent couplings over joint vertex factors. We call the elements in $\mathfrak{A}$ {\bf coupling structures}. Let $\mathfrak{S}$ denote the set of frames of all coupling structures. We call the elements of $\mathfrak{S}$ {\bf coupling frames}.
\end{definition} 

Notice that the fact that $f(G)$ is a probability distribution on $\Hom(H,G)$ implies that $\Hom(H,G)$ is not empty for every graph $G$. This shows that every graph in $\mathfrak{S}$ has to be bipartite. It follows from the definition that if a probability distribution on $\Hom(H,G)$ is constructed according to a probability scheme in $\mathfrak{A}$ then its marginals on the edges of $H$ are all identical to $\tau(e,G)$ and its marginals on the vertices are identical to $\kappa$.

Let $\mathfrak{A}_1\subset\mathfrak{A}$ be the subset in which only couplings over independent vertex sets are used. Correspondingly $\mathfrak{S}_1\subset\mathfrak{S}$ is the set of frames of the elements of $\mathfrak{A}_1$.
As an easy demonstration of our method we can immediately prove the following.

\medskip

\begin{proposition}\label{emptygluing} Every element in $\mathfrak{A}_1$ is a family of witness measures. Consequently every graph in $\mathfrak{S}_1$ satisfies the Sidorenko conjecture.
\end{proposition}

\begin{proof} It is trivial that $\tau(e,G)$ is a witness measure. Assume that $f_1$ and $f_2$ are probability schemes with  frames $H_1$ and $H_2$. Assume that $\{\beta_i:[n]\rightarrow V(H_i)\}_{i=1,2}$ defines a joint vertex factor such that the images of $\beta_1$ and $\beta_2$ are independent sets. Let $H$ be the frame of $g=C(f_1,f_2,\beta_1,\beta_2)$. Then from (\ref{pos}) and (\ref{inex1}) it follows that $$D(g)\leq D(f_1)+D(f_2)\leq (|E(H_1)|+|E(H_2)|)D_e=|E(H)|D_e.$$
\end{proof}

Proposition \ref{emptygluing} provides a very short unified proof for many results in the topic. In particular it implies that the so-called tree-arrangeable graphs introduced in \cite{KLL} satisfy Sidorenko's conjecture. Trees, reflection trees, even cycles and bipartite graphs in which one point is complete to the other side are all tree-arrangeable and thus we cover many results from the papers \cite{LS},\cite{Sid},\cite{BR},\cite{CFS},\cite{KLL}. 
Now we give a further strengthening of proposition \ref{emptygluing}.
Let $\mathfrak{A}_2\subset\mathfrak{A}$ be the set in which all couplings use vertex sets that span forests. Correspondingly $\mathfrak{S}_2\subset\mathfrak{S}$ is the set of frames of the elements of $\mathfrak{A}_2$.
We have that $\mathfrak{S}_1\subset\mathfrak{S}_2\subset\mathfrak{S}$.
We obtain the following result about Sidorenko's conjecture.

\medskip

\begin{theorem}\label{forestgluing} Every element in $\mathfrak{A}_2$ is a family of witness measures. Consequently every graph in $\mathfrak{S}_2$ satisfies the Sidorenko conjecture. 
\end{theorem}

We need the next lemma.

\begin{lemma}\label{forest} Let $H$ be a forest and $G$ be an arbitrary graph. Let $\mu$ be a probability measure on $\hom (H,G)$ such that the marginals on the edges of $H$ are identical with $\tau(e,G)$ and the marginals on the vertices are identical with $\kappa$. Then $$D(\mu)\geq D_e|E(H)|-D_v(2|E(H)|-|V(H)|).$$
\end{lemma}

\begin{proof} We go by induction. If $|V(H)|\leq 2$ then the statement is trivial. If $|V(H)\geq 3$ then there is a decomposition $V(H)=V_1\cup V_2$ such that $V_1\cap V_2$ is a single vertex $v$, there is no edge between $V_1$ and $V_2$ and $|V_1|,|V_2|<|V(H)|$. The measure $\mu$ is a couping of its marginals on $V_1$ and $V_2$ over its marginal on the vertex $v$. The induction hypothesis together with lemma \ref{nagyobb} completes the proof.
\end{proof}

Now we are ready to prove theorem \ref{forestgluing}.

\begin{proof} Assume that $f\in\mathfrak{A}_2$ is a probability scheme with frame $H$. We prove by induction that $D(f)\leq D_e|E(H)|-D_v(2|E(H)|-|V(H)|)$ which is clearly enough since $H$ has no isolated points and thus $2|E(H)|-|V(H)|$ can not be negative. It is trivial that $\tau(e,G)$ satisfies this inequality. Assume that $f_1$ and $f_2$ are probability schemes with frames $H_1$ and $H_2$. Assume that $\{\beta_i:[n]\rightarrow V(H_i)\}_{i=1,2}$ defines a joint vertex factor such that the images of $\beta_1$ and $\beta_2$ are identical forests. We call this forest $H_3$. Let $H$ be the frame of $g=C(f_1,f_2,\beta_1,\beta_2)$. Then from lemma \ref{forest}  and (\ref{inex1}) it follows that $$D(g)\leq D(f_1)+D(f_2)-(D_e|E(H_3)|-D_v(2|E(H_3)|-|V(H_3)|)).$$ Using the induction hypothesis the proof is complete.
\end{proof}

\bigskip

\section{Set functions and the genearal theorem}

\bigskip

Theorem \ref{forestgluing} is not the limitation of our method. In this chapter we describe a far reaching generalization of the idea in the proof of theorem \ref{forestgluing}.
We need some notation. 

Let $V$ be a finite set and let $H$ be a graph (or a hypergraph) with vertex set $V$. We will work in the linear space $\mathbb{R}^{2^V}$ of all set-functions on $V$. For $S\subseteq V$ we denote by $1_S$ the set-function that takes the value $1$ on $S$ and takes the value $0$ on any other subset in $V$. For a pair of sets $A,B\subseteq V$ let $$t_{A,B}:=1_{A\cup B}-1_A-1_B+1_{A\cap B}.$$ For an arbitrary binary relation $b\subseteq 2^V\times 2^V$ on the subsetes of $V$ let $$W_b:=\langle t_{A,B}~|~(A,B)\in b\rangle_{\mathbb{R}}.$$
 For a subset $A\in V$ let $s_H(A) := -1_A+\sum_{e\in E(A)}1_e$ where $E(A)$ is the set of edges in $H$ spanned by $A$. We denote by $Q_V$ the sum of the one dimensional space $\langle 1_{\emptyset}\rangle_{\mathbb{R}}$ and the cone (set of all non-negative linear combinations) spanned by the vectors $t_{A,B}$ and $1_A$ where $A,B$ runs through all possible pairs of subsets in $V$.

\medskip

Let $f$ be a coupling structure in $\mathfrak{A}$ with frame $H$. For a subset $A\subseteq V$ we denote by $f_A$ the restriction of $f$ to the coordinates in $A$. In other words $f_A$ is the vertex factor of $f$ with respect to the embedding of $A$ into $V$. If for every graph $G$ we have that $f_A$ and $f_B$ are conditionally independent over $f_{A\cap B}$ then we say that $(A,B)$ is a conditionally independent pair (with respect to $f$). We denote the set of all conditionally independent pairs by $CI(f)$. Let $IS(f)$ denote the set of pairs $(A,B)$ such that $f_A$ is isomorphic to $f_B$ i.e. there is a bijection $\beta:A\mapsto B$ such that $f_A=f|_\beta$. We denote by $C_f$ the linear space $W_{CI(f)}$ and by $I_f$ the space $\langle 1_A-1_B~|~(A,B)\in IS(f)\rangle_{\mathbb{R}}$.  

Using these notations we have the following general theorem.

\begin{theorem}\label{general} If $s_H(A)\in C_f+I_f+Q_V$ holds for some $A\subseteq V$ then the graph $(A,E(A))$ satisfies Sidorenko's conjecture.
\end{theorem}

\begin{proof}  Let $g$ be the set function defined by $g(B):=D(f_B(G))$ for $B\subseteq V(H)$. Using lemma \ref{nagyobb} we have that $g$ is a super modular function a thus $(g,q)\geq 0$ for every $q\in Q_V$. We have by (\ref{inex1}) that $(g,c)=0$ holds for every element in $c\in C_f$. If $(B_1,B_2)\in IS(f)$ then $g(B_1)=g(B_2)$. It follows that $(g,w)=0$ holds for every $w\in I_f$. We obtain that $(s_H(A),g)\geq 0$. This means that $g_A$ is a witness measure for $A$ and thus $(A,E(A))$ satisfies Sidorenko's conjecture.
\end{proof}

\medskip

Note that since $C_f+I_f+Q_V$ is a convex polytope it is a finite linear-programming problem to decide weather $s_A\in C_f+I_f+Q_V$ holds.

\bigskip

\section{Refelction complexes}\label{chaprefcomp}

\bigskip

%The definitions of the graph classes $\mathcal{S},\mathcal{S}_1$ and $\mathcal{S}_2$ are based on probability theory. It is very natural to ask for a more combinatorial description that is easier to verify for concrete graphs. The goal of this chapter is to give sufficient conditions for memebrships in these families that are completely combinatorial and can be checked in finite time. This will provide us with a powerful method for proving Sidorenko's conjecture for various concrete graphs. We believe that our conditions for the membership problem are also necessary however in this paper we don't investigate this problem

A hypergraph $M$ is a pair of a vertex set $V$ and edge set $E\subseteq 2^V$. We introduce the technical notion of a $b$-hypergraph which is a hypergraph $M$ together with a symmetric binary relation $B\subseteq 2^V\times 2^V$ on the subsets of $V$. The sub-$b$-hypergraph on $W\subseteq V$ is a $b$-hypergraph with edge set $E(W)=\{L~|~L\in E,L\subseteq W\}$ and relation $B(W)=(2^W\times 2^W)\cap B$.

\medskip

We describe two gluing operations for $b$-hypergraphs. (Both operations are meaningful for hypergraphs without a binary operation.) Assume that $M_1=(V_1,E_1,B_1)$ and $M_2=(V_2,E_2,B_2)$ are two $b$-hypergraphs and $\phi_1:F\rightarrow V_1, \phi_2: F\rightarrow V_2$ are injective maps for some label set $F$. We introduce a set $V$ together with injective maps $\tau_1:V_1\rightarrow V,\tau_2:V_2\rightarrow V$ such that
\begin{enumerate}
\item $\tau_1(v_1)=\tau_2(v_2)$ if and only if there is $f\in F$ with $\phi_1(f)=v_1$ and $\phi_2(f)=v_2$.
\item $V=\tau_1(V_1)\cup\tau_2(V_2)$.
\end{enumerate}
Note that there is a natural embedding $\phi:F\rightarrow V$ defined by $\phi=\tau_1\circ\phi_1=\tau_2\circ\phi_2$. We define the edges set $E$ of a new $b$-hypergraph denoted by $M_1\cup_{\phi_1,\phi_2} M_2$ with vertex set $V$ as $\tau_1(E_1)\cup\tau_2(E_2)$ and the set $B$ of binary relations on $2^V$ as $\tau_1(B_1)\cup\tau_2(B_2)$. Informally speaking, $M_1\cup_{\phi_1,\phi_2} M_2$ is obtained by first taking a disjoint copy of $M_1$ and $M_2$ and then we identify vertices with the same label. We will need another construction denoted by $M_1\cup^*_{\phi_1,\phi_2} M_2$ that is obtained from $M_1\cup_{\phi_1,\phi_2} M_2$ by extending the edge set end the binary relations.  We add $(\tau_1(K_1),\tau_2(K_2))$ to $B$ and $\tau_1(K_1)\cup\tau_2(K_2)$ to $E$ for all pairs $K_1\in E_1,K_2\in E_2$ with $\phi_1(F)\subseteq K_1,\phi_2(F)\subseteq K_2$. When talking about $M_1\cup^*_{\phi_1,\phi_2} M_2$ we will refer to $E$ (resp. $B$) as edges (resp. relations) of the first type and we call the remaining (added) edges (resp. relations) edges (resp. relations) of the second type. 

\medskip

Let $M=(V,E,B)$ be a $b$-hypergraph. Let $L\in E$ be an edge and $N$ be the sub-$b$-hypergraph on $L$. Let $X\subseteq L$ be some set and let $\phi_1:X\rightarrow V$ and $\phi_2:X\rightarrow L$ be the identical embedding maps. We will use the notation  $r_{L,X}(M)$ for the $b$-hypergraph $M\cup^*_{\phi_1,\phi_2} N$. 

\medskip

\begin{definition}\label{refcomp}  Let $M_0=(V_0,E_0,B_0)$ be a $b$-hypergraph. A $b$-hypergraph $M$ is called an {$M_0$-\bf reflection complex} if there is a sequence of $b$-hypergraphs $\{M_i=(V_i,E_i,B_i)\}_{i=1}^n$ and pairs $\{L_i\in E_i,X_i\subseteq L_i\}_{i=0}^{n-1}$ such that $M_i=r_{L_{i-1},X_{i-1}}(M_{i-1})$ holds for $1\leq i\leq n$ and $M=M_n$.  If $M_0$ is the $b$-hypergraph with vertex set $V_0=\{1,2\}$, edge set $E_0=\{\{1,2\}\}$ and empty relation $B_0$ then $M$ will be called a {\bf reflection complex} and $M_0$ will be called the {\bf trivial} reflection complex.
\end{definition}

In other words $M$ is a reflection complex if it can be obtained from the trivial reflection complex by a finite sequence of operations of type $r_{L,X}$. It follows from this definition by induction that the full vertex set of $M$ is an edge of $M$.

\begin{definition} We say that a $b$-hypergraph $M=(V,E,B)$ is {\bf $k$-reducible} if $V\in E$ and every edge $T$ of $M$ with $|T|>k$ ha a proper decomposition i.e. $T=A_1\cup A_2$ for some edges $A_1$ and $A_2$ of $M$ with $\max(|A_1|,|A_2|)<|T|$ and $(A_1,A_2)\in B$. 
\end{definition}

\begin{lemma}\label{reducible} Every reflection complex is $2$-reducible.
\end{lemma}

\begin{proof} Let $n$ be the number of vertices of a reflection complex $N$. We prove the statement by induction on $n$. The case $n=2$ is trivial. If $n>2$ we can assume by induction that $N=r_{L,X}(M)$ for some reflection complex $M=(V,E)$ that satisfies the theorem. We use the notation from the definition of $r_{L,X}$ and $\cup^*$. If $T$ is of the first type then the induction hypothesis guarantees the decomposition. If $T$ of the second type then its decomposition $T=\tau_1(K_1)\cup\tau_2(K_2)$ guaranteed by the definition of $\cup^*$ is a proper decomposition.   
\end{proof}

\begin{definition} For a general hypergraph $M$ with edge set $E$ we say that the set $\mathcal{F}(M):=\{K\in E~|~|K|=2\}$ is the {\bf frame of $M$}. 
\end{definition}

\begin{lemma}\label{frun} Let $M=(V,E,B)$ be a reflection complex with $L\in E, X\subseteq L$. Then $\mathcal{F}(r_{L,X}(M))=\tau_1(\mathcal{F}(M))\cup\tau_2(\mathcal{F}(N))$ where $N$ is the $b$-hypergraph spanned on $L$.
\end{lemma}

\begin{proof} Observe that every edge of a reflection complex is of size at least two and so edges of the second type in the construction of $r_{L,X}(M)$ have size at least three. This means that edges of the second type don't contribute to $\mathcal{F}(r_{L,X}(M))$.
\end{proof}

\section{Thick reflection complexes and graphs}

Let $H$ be a graph on the vertex set $V$ and let $X\subseteq V$. We introduce the set function $$h_H(X):=-1_X+\sum_{e\in E(F)}1_e-\sum_{v\in X}1_v(\deg(v)-1)$$ where $F$ is the graph spanned on $X$ in $H$ and $\deg(v)$ is the degree of $v$ in $F$.

\begin{definition}\label{defthick} Let $M=(V,E,B)$ be a $b$-hypergraph with $H=\mathcal{F}(M)$. We say that $M$ is {\bf thick} if $$h_H(V)\in W_B+Q_V.$$ Frames of thick reflection complexes are called {\bf thick graphs}. 
\end{definition}

\bigskip

The name thick refers to the fact that the cone $W_B+Q_V$ (which is a convex polytope) is large enough to contain the vector $h_H(V)$. Observe that if $M$ is the trivial reflection complex then $h_V=0$ and thus the trivial reflection complex is thick. 
The next theorem will be proved in a chapter \ref{rccs}.

\begin{theorem}\label{thicksid} Thick graphs satisfy Sidorenko's conjecture.
\end{theorem} 

\begin{remark} Potentially we could replace thickness with the seemingly weaker condition that $s_H(V)\in W_B+Q_V$. Let us call such graphs {\bf weakly thick}. This notion is more in the spirit of theorem \ref{general} and it would still imply Sidorenko's conjecture for $H$. However we don't know any graph that is weakly thick but not thick. Secondly, thickness behaves better with respect to certain operations than weak thickness. The notion of weak thickness will be used later when we work with hypergraphs. 
\end{remark}

\begin{remark} In theorem \ref{general} there is a term responsible for isomorphic pairs of subsets. This notion can also be interpreted for reflection complexes and could lead to a more general sufficient condition for Sidorenko's conjecture. We don't know any concrete example where this seemingly useful term helps. 
\end{remark}

Despite of the fact that the definition of thickness uses linear algebra we will introduce combinatorial operations that preserve this property. These operations will help us to prove the property for large classes of graphs. The next lemma follows directly from the definition of $h_H(X)$. 

\begin{lemma}\label{hincl} Let $H=(V,E)$ be a graph. Then $$h_H(A_1\cup A_2)=h_H(A_1)+h_H(A_2)-h_H(A_1\cap A_2)-t_{A_1,A_2}$$ holds for every pair $A_1,A_2\subseteq V$.
\end{lemma}

We will also need the following two lemmas.

\begin{lemma}\label{forcon} Let $H=(V,E)$ be a forest. Then $-h_H(V)\in Q_V$. 
\end{lemma}

\begin{proof} We prove the statement by induction on $|V|$. If $|V|=1$ or $|V|=2,|E|=1$ then $h_H(V)=0$. In every other case there is a decomposition $V=V_1\cup V_2$ such that $|V_1|<|V|~,~|V_2|<|V|~,~|V_1\cap V_2|\leq 1$ and $E=E(V_1)\cup E(V_2)$. We have by lemma \ref{hincl} that $-h_H(V)$ is equal to $-h_H(V_1)-h_H(V_2)+h_H(V_1\cap V_2)+t(V_1,V_2)$. The first two terms are in $Q_V$ by induction. The last term is in $Q_V$ by definition. The term  $h_V(V_1\cap V_2)$ is $0$ if $|V_1\cap V_2|=1$ and is $1_\emptyset\in Q_V$ if $V_1\cap V_2=\emptyset$.
\end{proof}

\begin{lemma}\label{genforgl} Assume that $M_1=(V_1,E_1,B_1)$ and $M_2=(V_2,E_2,B_2)$ are two thick $b$-hypergraphs and $\phi_1:F\rightarrow V_1, \phi_2: F\rightarrow V_2$ are injective maps for some set $F$. Assume that $\phi_2\circ\phi_1^{-1}$ is an isomorphism between the graph $H$ spanned on $\phi_1(F)$ in $\mathcal{F}(M_1)$ and the graph on $\phi_2(F)$ in $\mathcal{F}(M_2)$. Assume furthermore that $H$ is a forest. Then $M=(V,E,B)=M_1\cup^*_{\phi_1,\phi_2} M_2$ is also thick. 
\end{lemma}

\begin{proof} Let $A_1=\tau_1(V_1)$ and $A_2=\tau_2(V_2)$ using the notation from the definition of $\cup^*$. Let $H=\mathcal{F}(M)$. 
Using that $t_{A_1,A_2}\in W_B$ we have by lemma \ref{hincl} and lemma \ref{forcon} that $h_H(V)\in h_H(V_1)+h_H(V_2)+Q_V+W_B$.  Using the assumption that $M_1$ and $M_2$ are thick we obtain that $h_H(V_1)+h_H(V_2)\in W_B+Q_V$ and thus $h_H(V)\in Q_V+W_B$.
\end{proof}

\begin{definition} Let us denote the set of reflection complexes $M$ by $\mathcal{C}$ that can be obtained from the trivial reflection complex by a sequence of operations of type $N\mapsto r_{L,X}(N)$ where $X$ spans a forest in $\mathcal{F}(N)$. Let us denote the set of graphs that are frames of some reflection complex in $\mathcal{C}$ by $\mathcal{R}$. 
\end{definition}

\begin{theorem}\label{genforgl2} Let $M=(V,E,B)$ be a reflection complex in $\mathcal{C}$. Then every edge in $M$ spans a thick $b$-hypergraph. In particular $M$ is a thick reflection complex. 
\end{theorem}

\begin{proof} The theorem is true for the trivial reflection complex. We go by induction on $|V|$. Assume that $M=(V,E,B)$ is a reflection complex in $\mathcal{C}$ and $N=r_{L,X}(M)$ for some $L\in E$, $X\subseteq L$ such that $X$ spans a forest in $\mathcal{F}(M)$.
The statement is trivial for edges in $N$ of the first type. If $K=\tau_1(K_1)\cup\tau_2(K_2)$ is an edge of the second type then lemma \ref{genforgl} implies that $K$ spans a thick $b$-hypergraph.
\end{proof}

\begin{corollary} Every graph in $\mathcal{R}$ is thick. 
\end{corollary}

\section{Subdivisions of graphs and reflection complexes}

Let $H_1=(V_1,E_1)$ and $H_2=(V_2,E_2)$ be bipartite graphs with bipartition $m:V_2\rightarrow\{1,2\}$ of $H_2$. Assume furthermore that $J_1,J_2$ are two disjoint subsets in $V_1$. Then we define the $\{H_1,J_1,J_2\}$ subdivision of $H_2$ in the following way. We blow up every point $v$ of $V_2$ into a copy of the set $J_{m(v)}$ (all disjoint) called $J'_v$ and we replace every edge $(v_1,v_2)\in E_2$ by a copy of $H_1$ such that $J_1$ is glued on $J'_{v_1}$ and $J_2$ is glued on $J'_{v_2}$ using the natural bijection. Note that if $H_2$ is connected then there are two bipartitions of $V_2$ and thus there are two subdivisions. The main theorem of this chapter is the following.

\begin{theorem}\label{subdiv} Let $H_1=(V_1,E_1)$ and $H_2=(V_2,E_2)$ be thick graphs and assume that $J_1,J_2$ are disjoint subsets in $V_1$ such that both $J_1,J_2$ span a forest in $H_1$. Then the $\{H_1,J_1,J_2\}$ subdivision of $H_2$ is also a thick graph. 
\end{theorem}

The following product notion for graphs was studied in \cite{KLL}. For two graphs $H_1=(V_1,E_1),H_2=(V_2,E_2)$ the vertex set of $H_1\square H_2$ is $V_1\times V_2$. Two vertices $(v_1,v_2),(w_1,w_2)$ are connected if either $v_1=w_1,~(v_2,w_2)\in E_2$ or $v_2=w_2,~(v_1,w_1)\in E_1$. Let $e=(V_e,E_e)$ be the single edge with $V_e=\{1,2\},E_e=\{\{1,2\}\}$. It is clear that $H_1\square H_2$ is a subdivision of $H_2$ with $H_1\square e$ such that $J_1=V_1\times\{1\}$ and $J_2=V_1\times\{2\}$. It is easy to see (and we will show it in the chapter of examples) that if $H_1$ is a tree then $H_1\square e$ is in the class $\mathcal{R}$ and so it is a thick graph. Using this we obtain the following corollary of of theorem \ref{subdiv}.

\begin{corollary}\label{product} If $H_1$ and $H_2$ are two graphs such that $H_1$ is a tree and $H_2$ is thick then $H_1\square H_2$ is thick.
\end{corollary}

It was proved in \cite{KLL} that the family of graphs satisfying Sidorenko's conjecture is closed with respect to taking $\square$-product with trees. Corollary \ref{product} says the same thing for thick graphs.  

The rest of this chapter is the proof of theorem \ref{subdiv}. As a preparation we need to extend the notion of subdivision to reflection complexes. Let $N=(V',E',B')$ be a reflection complex with two distinguished disjoint sets $J_1,J_2\subset V'$. If $M=(V,E,B)$ is some reflection complex built up from the trivial reflection complex $M_0$ with a sequence of operations $\{r_{L_i,X_i}\}_{i=0}^{n-1}$ then we can repeat essentially the same operations starting from $N$ instead of $M_0$ in a way that we use $J_1$ and $J_2$ instead of the points $1$ and $2$ in $M_0$. In the resulting reflection complex every point of $M$ will be blown up into either $J_1$ or $J_2$ and every edge in $\mathcal{F}(M)$ will be replace by a copy of $N$. Our first goal is to make this construction precise. 

We use the notation from definition \ref{refcomp}. Using the sequence $\{M_i\}_{i=0}^n$ we construct a new sequence $\{\hat{M}_i=(\hat{V_i},\hat{E_i},\hat{B_i})\}_{i=0}^n$ in a recursive way together with functions $\{\gamma_i:2^{V_i}\rightarrow 2^{\hat{V_i}}\}_{i=0}^n$ such that $\gamma_i$ maps edges of $M_i$ to edges of $\hat{M_i}$. For $i=0$ we set $\hat{V_0}=V',\hat{E_0}=E',\hat{B_0}=B'$ and $\gamma_0(1)=J_1,\gamma_0(2)=J_2,\gamma_0(\{1,2\})=\hat{V}_0,\gamma_0(\emptyset)=\emptyset$. Assume that $\hat{M}_{i-1},\gamma_{i-1}$ is already constructed. Then we set $\hat{L}_{i-1}:=\gamma_{i-1}(L_{i-1}),\hat{X}_{i-1}:=\gamma_{i-1}(X_{i-1})$ and $$\hat{M_i}:=r_{\hat{L}_{i-1},\hat{X}_{i-1}}(\hat M_{i-1}).$$ We define $\gamma_i$ in the following way. If $S\subset V_i$ is contained in one of $\tau_1(V_{i-1})$ or $\tau_2(L_{i-1})$ then $\gamma_i(S)$ is defined as the corresponding copy of $\gamma_{i-1}(S)$ in $\tau_1(\hat{V}_{i-1})$ or $\tau_2(\hat{L}_{i-1})$. The gluing guarantees the consistency of this definition for sets contained in both sets. For a general set $S\subseteq V_i$ we set $$\gamma_i(S):=\gamma_i(S\cap\tau_1(V_{i-1}))\cup\gamma_i(S\cap\tau_2(L_{i-1})).$$ We say that $\hat{M_n}$ is the subdivision of $M=M_n$ with $N$. Not that this definition depends on the way we build up $M$ however the frame $\mathcal{F}(\hat{M}_n)$ is always the subdivision of $\mathcal{F}(M)$ with $\{\mathcal{F}(N),J_1,J_2\}$.

Using the above notation we observe the following facts. The map $\gamma_i$ satisfies $\gamma_i(A_1)\cup\gamma_i(A_2)\subseteq \gamma_i(A_1\cup A_2)$ for $A_1,A_2\subseteq V_i$ and we have equality if $(A_1,A_2)\in B_i$. Furthermore $\gamma_i(A_1\cap A_2)=\gamma_i(A_1)\cap\gamma_i(A_2)$ holds for every pair $A_1,A_2\subseteq V_i$. These statements follow trivially by induction. Let us introduce the linear map $\gamma^*$ from the space of setfunctions on $V_n$ to the space of setfunctions on $\hat{V}_n$ defined as the linear extension of the map $\gamma^*(1_A):=1_{\gamma_n(A)}$. If $(A_1,A_2)\in B_n$ then $\gamma^*(t_{A_1,A_2})=t_{\gamma_n(A_1),\gamma_n(A_2)}$ and thus $\gamma^*(W_{B_n})\subseteq W_{\hat{B}_n}$. For two arbitrary sets $A_1,A_2\subseteq V_n$ let $C=\gamma_n(A_1\cup A_2)\setminus(\gamma_n(A_1)\cup\gamma_n(A_2))$. We have that $$\gamma^*(t_{A_1,A_2})=t_{\gamma_n(A_1),\gamma_n(A_2)}+t_{C,\gamma_n(A_1)\cup\gamma_n(A_2)}+1_C-1_\emptyset\in Q_{\hat{V}_n}.$$ This implies that $\gamma^*(Q_{V_n})\subseteq Q_{\hat{V}_n}$. 

Now we prove that if $M$ (recall that $M=M_n$) and $N$ are thick and $J_1,J_2$ span forests in $\mathcal{F}(N)$ then the reflection complex $\hat{M}_n$ is also thick. This completes the proof of theorem \ref{subdiv}. 

Let $P:= W_{\hat{B}_n}+Q_{\hat{V}_n}$ and let $H=\mathcal{F}(\hat{M}_n)$. By assumption we have that $h=h_{\mathcal{F}(M_n)}(V_n)\in W_{B_n}+Q_{V_n}$. Using the previous observations we get that $\gamma^*(h)\in P$. By the assumption that $N$ is thick we obtain that $h_H(\gamma_n(A))$ holds for every $A\in\mathcal{F}(M_n)$ . Furthermore, since $\gamma_n(\{v\})$ spans a forest in $H$ for every $v\in V_n$ we have by lemma \ref{forcon} that $-h_H(\gamma_n(\{v\})\in P$.  We obtain that $$h_H(\hat{V}_n)=\gamma^*(h)+\sum_{A\in\mathcal{F}(M_n)}h_H(\gamma_n(A))
+\sum_{v\in V}(1-\deg{v})h_H(\gamma_n(\{v\})$$ is in $P$.  This shows that $\hat{M}_n$ is thick.

\section{Reflection complexes and coupling structures}\label{rccs}

\begin{lemma}\label{umeasure} Let $M=(V,E,B)$ be a  $2$-reducible $b$-hypergraph and let $G$ be a finite graph with $|E(G)|>0$. Then there is at most one probability distribution  $\mu$ on $V(G)^M$ such that
\begin{enumerate}
\item If $(A_1, A_2)\in B$ for edges $A_1,A_2\in E$ then $\mu_{A_1}$ and $\mu_{A_2}$ are conditionally independent over $\mu_{A_1\cap A_2}$,
\item $\mu_A$ is a uniform random edge in $G$ for every $A\in E$ with $|A|=2$.
\end{enumerate}
\end{lemma}

\begin{proof} We prove the statement by induction on $|V|$. If $|V|=2$ then the statement is trivial. Since every edge in  a $2$-reducible $b$-hypergraph spans a $2$-reducible $b$-hypergraph we can assume by induction that the statement is true for all edges in $M$ of size smaller that $|V|$. Let $V=A_1\cup A_2,~(A_1,A_2)\in B,~A_1,A_2\in E$ be a proper decomposition. Then any probability distribution $\mu$ satisfying the requirements is a conditionally independent coupling of $\mu_{A_1}$ and $\mu_{A_2}$ over $\mu_{A_1\cap A_2}$. Using that $\mu_{A_1}$ and $\mu_{A_2}$ are unique we have that $\mu$ (if there is such a measure at all) is unique.
\end{proof}

\begin{theorem}\label{refcoup} Let $N$ be a reflection complex with frame $H=\mathcal{F}(N)$. Then there is a unique coupling structure $f\in\mathfrak{A}$ such that
\begin{enumerate}
\item The frame of $f$ is $H$,
\item $f_e(G)$ is a uniform random edge in $G$ for every finite graph graph $G$ and edge $e$ in $H$,
\item every pair $(A_1,A_2)\in B(N)$ is a conditionally independent pair of $f$. 
\end{enumerate}
Furthermore we have that $f_A\in\mathfrak{A}$ holds for every edge $A$ in $N$.
\end{theorem}

\begin{proof}  It is enough to show the existence of $f$ since the uniqueness follows from lemma \ref{reducible} and lemma \ref{umeasure}. Let $n$ be the number of vertices of $N$. We prove the statement by induction on $n$. The case $n=2$ is trivial. If $n>2$ we can assume by induction that $N=r_{L,X}(M)$ for some reflection complex $M=(V,E,B)$ that satisfies the theorem with coupling structure $f'$. For a fix graph $G$ we have that $f'(G)$ is a probability distribution on $V(G)^V$. We construct $f(G)$ as the conditionally independent coupling of $f'(G)$ with an identical copy of its marginal distribution on the coordinates in $L$ over the joint factor given by the marginal distribution on the coordinates in $X$. We have that the probability distribution $f'(G)$ is automatically defined on the vertex set $V(N)$ of $N$. 

We check the statements of the theorem for $f$. For the first types of edges and relations (see the definition of $r_{L,X}$) all the statements follows directly from the fact that $f'$ satisfies the statement. In particular the second condition remains valid for $f$.
 For edges and relations of the second type assume that $K_1,K_2\in E, K_2\subseteq L$ and $X\subseteq K_1,K_2$. Observe that since $f$ is the conditionally independent coupling of $f_{\tau_1(V)}$ and $f_{\tau_2(L)}$ over $f_{\tau(X)}$ we have that $f_{\tau_1(K_1)}$ and $f_{\tau_2(K_2)}$ are conditionally independent over $f_{\tau(X)}$ and thus $f_K$ is the conditionally independent coupling of two probability schemes in $\mathfrak{A}$ over a joint vertex factor. It follows that the marginal on $K$ is in $\mathfrak{A}$. Furthermore this verifies conditional independence for the pair $(\tau_1(K_1),\tau_2(K_2))\in B(N)$. The statement on the frame follows from lemma \ref{frun}.

\end{proof}

\begin{corollary} Frames of reflection complexes are coupling frames. In other words $\mathcal{F}(N)\in\mathfrak{S}$ holds for every reflection complex.
\end{corollary}

\medskip

We are ready to prove that thick graphs satisfy Sidorenko's conjecture.

\noindent{\it Proof of Theorem \ref{thicksid}.}~~If $H$ is a thick graph then $H=\mathcal{F}(M)$ for some thick reflection complex $M=(V,E,B)$. We have that $h_H(V)\in W_B+Q_V$ . Since all components of $s_H(V)-h_H(V)$ are positive we get that $s_H(V)\in W_B+Q_V$. According to Theorem \ref{refcoup} there is a coupling structure $f\in\mathfrak{A}$ with frame $H$ such that $(A_1,A_2)\in B$ implies that $(A_1,A_2)\in CI(f)$. Then theorem \ref{general} implies that $H$ satisfies Sidorenko's conjecture. 

\bigskip

\section{Hypergraphs}

Most of the results in this paper have a generalization to $k$-uniform hypergraphs. Coupling structures can be defined exactly in the same way as for graphs. Our general theorem \ref{general} works exactly the same way as for graphs. We need a few alterations in the results about reflection complexes. Here is the list of these alterations.
\begin{enumerate}
\item When we work with $k$-uniform hypergraphs, the trivial $k$-reflection complex is defined to be the single $k$-edge $\{1,2,\dots,k\}$ with the trivial binary relation. A $k$-reflection complex is built up from the trivial $k$-reflection  complex by operations of the form $r_{L,X}$. 
\item The frame $\mathcal{F}(M)$ of a $k$-reflection complex $M$ is the set of edges of size $k$. Note that in a $k$-reflection complex there are no smaller edges.
\item Lemma \ref{reducible} has to be changed to the statement that $k$-reflection complexes are $k$-reducible. The proof is basically identical.
\item We say that a $k$-reflection complex $M=(V,E,B)$ is weakly thick if $s_H(V)\in W_B+Q_V$ where $H=\mathcal{F}(M)$ is the frame of $M$. A $k$-uniform hypergraph is called weakly thick if it is the frame of a weakly thick $k$-reflection complex $M$. 
\item In lemma \ref{umeasure} $2$-reducible has to be replaced by $k$-reducible and in the third condition $|A|=2$ has to be replaced by $|A|=k$. 
\end{enumerate}

With these alterations we have the following theorem.

\begin{theorem} Every weakly thick $k$-uniform hypergraph satisfies Sidorenko's conjecture. 
\end{theorem}

In the rest of this chapter we construct a class of weakly thick $k$-uniform hypergraphs. Let $\mathcal{Z}_k$ denote the set of $k$-uniform hypergraphs consisting of isolated (non-intersecting edges) and isolated points. The key idea is that the analogy of lemma \ref{forcon} holds with $s_H(V)$ inside the family $\mathcal{Z}_k$. More perceisely, if $H=(V,E)$ is a $k$-uniform hypergraph in $\mathcal{Z}_k$ then $-s_H(V)\in Q_V$. The proof is similar to the proof of lemma \ref{forcon} but even simpler. We go by induction on $|V|$. The statement is trivial if $H$ is a single point or a single $k$-edge. In any other case $H$ is the disjoint union of two smaller hypergraphs in $\mathcal{Z}_k$ and then the same calculation finishes the proof as in lemma \ref{forcon}.

\begin{definition} Let $\mathcal{C}_k$ denote the set of $k$-reflection complexes that can be built up from the trivial $k$-reflection complex using operations of the form $M\mapsto r_{L,X}(M)$ where $X$ spans a hypergraph in $\mathcal{F}(M)$ that is in $\mathcal{Z}_k$. Let $\mathcal{R}_k$ denote the frames of $k$-reflection complexes in $\mathcal{C}_k$. 
\end{definition}

Using the same arguments as in lemma \ref{genforgl} and in theorem \ref{genforgl2} with graphs in $\mathcal{Z}_k$ instead of forests we obtain the following.

\begin{theorem}\label{hypsid} Hypergraphs in $\mathcal{R}_k$ are weakly thick and thus they satisfy Sidorenko's conjecture.
\end{theorem}

Note that $\mathcal{R}_2$ is a smaller family than $\mathcal{R}$ but it contains trees, even cycles, bipartite graphs in which one point is complete to the other side and even tree-arrangeable graphs. For this reason theorem \ref{hypsid} can be considered as a hypergraph generalization of the famous Blakley-Roy inequality and the result by Conlon-Fox and Sudakov. 
The family $\mathcal{R}_k$ contains hypergraph analogues of forests that will be described in the examples chapter.

\section{Examples}

The class of thick graphs is a large class of graphs satisfying Sidorenko's conjecture. In this chapter we show how previous results on Sidorenko's conjecture fit into this framework. Quite interestingly, even the subclass $\mathcal{R}$ of thick graphs covers most of the previous cases. At the end of the chapter we show a few examples for the hypergraph case including a generalization of the Blakley-Roy inequality.

\bigskip

\noindent{\bf Trees}~~Let us build a reflection complex $M=(V,E,B)$ by iterating operations of the form $r_{\{a,b\},\{b\}}$ starting from the trivial reflection complex. By induction we have that $M\in\mathcal{C}$ and so $M$ is thick. The frame of $M$ is a tree $T$ and $E$ consists of the subsets in $V$ that span a connected sub-graph (a tree) in $T$. It is clear by induction that every tree can be obtained in this way and so every tree is a thick graph in $\mathcal{R}$.
\bigskip

\noindent{\bf Reflection trees}~~Let $T$ be a tree and $M$ be a reflection complex in $\mathcal{C}$ with frame $T$ as in the previous exampla.  Consider a sequence of edges $E_1,E_2,\dots,E_k$ in $M$ and subsets $X_1,X_2,\dots,X_k$ in $V$ with $X_i\subseteq E_i$ for $1\leq i\leq k$. Let us apply all the operations $r_{E_i,X_i}$ to $M$ in an arbitrary order to obtain a new reflection complex $M'$. We have that $M'\in\mathcal{C}$. The frame $\mathcal{F}(M')$ of the resulting reflection complex $M'$ is a so-called reflection tree. Reflection trees are all in the class $\mathcal{R}$ and so they are thick graphs. Reflection trees were introduced in \cite{LS} and Sidorenko's conjecture was verified for them. Note that reflection trees include all bipartite graphs in which one point is complete to the other side. Sidorenko's conjecture was first verified for such graphs by Conlon, Fox and Sudakov in \cite{CFS}. Reflection trees also include all even cycles.

\bigskip

\noindent{\bf Tree-arrangeable graphs}~~Let $\mathcal{T}$ be the class of bipartite graphs that can be built up from the single edge using the following two operations 1.) We add a new vertex to the second partition class connected to a vertex in the first partition class, 2.) we add a vertex to the first color class and connect it to a subset of the neighbors of another vertex in the first color class. We show by induction that $\mathcal{T}$ is contained in $\mathcal{R}$. The idea is that in the same way as we build up $H\in\mathcal{T}$ we can build up a reflection complex in $\mathcal{C}$. The first operation is represented by $r_{\{v,w\},\{v\}}$ where $v$ is in the first partition class and $\{v,w\}$ is in the frame of the reflection complex. The second operation is represented by $r_{\{v\}\cup S,S}$ where $v$ is in the first partition class and $S$ is a subset of its neighbors in the frame of the reflection complex. It remains to show that the sets of the form $\{v\}\cup S$ are always hyper-edges in the reflection complex that  we build up. This is clear by induction since both operations preserve this property.  Graph in $\mathcal{T}$ are called tree-arrangeable. It was proved in \cite{KLL} that tree-arrangeable graphs satisfy Sidorenko's conjecture. We obtained that they are all thick graphs.

\bigskip

\noindent{\bf Hypercubes and grids}~~The $n$-dimensional hypercube is the graph on the vertex set $\{0,1\}^n$ in which two $0-1$ sequences are connected if their Hamming distance is $1$. It was proved by Hatami \cite{Hat} that hypercubes satisfy Sidorenko's conjecture in every dimension. Since hypercubes are all of the form $e\square e\square\dots\square e$ we have by corollary \ref{product} that hypercubes are all thick graphs. An interesting fact is that hypercubes up to dimension $5$ are in the class $\mathcal{R}$. We conjecture that in high enough dimension they are not in $\mathcal{R}$ and thus the class of thick graphs is bigger then $\mathcal{R}$. 

A grid (in dimesion $n$) is a graph of the porm $P_1\square P_2\square\dots\square P_n$ where each $P_i$ is a path. It follows from the results in \cite{KLL} that all the grids satisfy Sidorenko's conjecture. It follows from corollary \ref{product} that grids are all thick graphs. 

\bigskip

\noindent{\bf Bipartate graphs with at most $4$ vertices on one side}~~It was proved by Sidorenko that bipartite graphs with at most $4$ points on one side satisfy Sidorenko's conjecture. Unfortunately the proof can't be find in the paper \cite{Sid} and the paper that supposed to contain the proof is only available in Russian. It turns out that all of these graphs are contained in the family $\mathcal{R}$ and thus they are all thick. The proof is a tedious case analysis and each case relies on a tricky way of building up a reflection complex. Here we only prove the statement for the case when one side has at most $3$ points. Let $H=(V,E)$ be a bipartite graph with bipartition $V=A\cup B$. We can represent $H$ by a function $m_H:2^A\rightarrow\mathbb{N}$ in a way that for $s\subseteq A$ the value of $m_H(S)$ is the number of points in $B$ whose set of neighbors is $S$. Graphs represented by the same function are isomorphic to each other. The family $\mathcal{R}$ has the property that we can glue single edges to each point an thus it is enough to treat those cases when $m_H(\{a\})=0$ holds for every $a\in A$. Furthermore if $m_H(A)\neq 0$ then one point of $B$ is complete to the other side and thus $H\in\mathcal{R}$. These observations cover the case of $|A|\leq 2$. If $|A|=3$ then we are left with the case when each point of $B$ is connected to exactly two points in $A$. Assume that $A=\{1,2,3\}$. Then $H$ is represented by the vector $v=(m_H(\{2,3\},m_H(\{1,3\}),m_H(\{1,2\}))$. If there is a zero coordinate in $v$ then $H$ is a reflection tree and so we are done. In the remaining case we do the following. Let us consider $C_6$ on the vertex set $W=\{1,2,3,z_1,z_2,z_3\}$ such that $(i,z_j)$ is an edge if and only if $i\neq j$. The graph $C_6$ is in $\mathcal{R}$ and so there is a reflection complex in $\mathcal{C}$ with this frame. Then we apply $r_{W,W\setminus\{z_j\}}$ to $M$ $v_j-1$ times for $j=1,2,3$. We obtain $H$ as the frame of the resulting reflection complex that is in $\mathcal{C}$.

\bigskip

\noindent{\bf Hypergraph forests and other hypergraph examples} 

\bigskip

Let us define the analogue of a forest in the $k$-uniform hypergraph setting in the following way. A single $k$-edge is a $k$-forest. If $H=(V,E)$ is a $k$-forest and $L\in E$ is an edge with a subset $X\subseteq L$ then the hypergraph $H \cup_{\phi_1,\phi_2} L$ is also a $k$-forest where $\phi_1$ and $\phi_2$ are the embeddings of $X$ into $V$ and $L$. (recall that the operation $\cup_{\phi_1,\phi_2}$ is meaningful for hypergraphs without a binary operation) Note that this definition does not allow independent points and so the $k=2$ case does not completely match the usual notion of a forest. 
It follows from theorem \ref{hypsid} that $k$-forests satisfy Sidorenko's conjecture. 
In particular we obtain the following generalization of the Blakely-Roy inequality. Let $f(x_1,x_2,\dots,x_k)$ by a symmetric non-negative function on $[0,1]^k$. Then for every $n\geq k$ we have that
\begin{equation}\label{hyproy}
\int_{[0,1]^n}\prod_{i=1}^{n-k+1}f(x_i,x_{i+1},\dots,x_{i+k-1}) ~d\mu^n\geq\Bigl(\int_{[0,1]^k}f(x_1,x_2,\dots,x_k)~d\mu^k\Bigr)^{n-k+1}.
\end{equation}
Note that the class $\mathcal{R}_k$ contains many more hypergraphs that satisfy Sidorenko's conjecture. A nice $3$-uniform example in $\mathcal{R}_3$ is the face-hypergraph of the octahedron. A generalization of this example is the complete $k$-uniform, $k$-partite hypergraph $K_{a_1,a_2,\dots,a_k}$ that is also in $\mathcal{R}_k$.
The class of weakly thick $k$-uniform hypergraphs is probably much bigger that $\mathcal{R}_k$ and it seems to be an interesting research topic to give a completely combinatorial description of it. 

\bigskip

\noindent{\bf Acknowledgement:}~The author is grateful to M\'ari\'o Szegedy for helpful comments. This research was supported by the MTA R\'enyi Institute Lendulet Limits of Discrete Structures group.

\end{document}